\documentclass[12pt,reqno]{amsart}
\usepackage{amsfonts}
\usepackage{amsthm}
\usepackage{amsmath}
\usepackage{amscd}
\numberwithin{equation}{section}
\usepackage{hyperref}
\hypersetup{nesting=true,debug=true,naturalnames=true}
\usepackage{graphicx,amssymb,upref}
\usepackage{enumerate}
\usepackage[all]{xy}
\usepackage{color}
\usepackage{mathrsfs}
\usepackage{caption}
\usepackage{paralist}
\usepackage{color}
\usepackage{float}
\usepackage{tikz,pgf}
\usetikzlibrary{calc}
\usetikzlibrary{intersections,decorations.markings}
\usepackage{tikz-cd}
\usepackage{lineno}

\theoremstyle{definition}
\newtheorem{theorem}{\bf Theorem}[section]

\newtheorem{lemma}[theorem]{\bf Lemma}

\theoremstyle{definition}
\newtheorem{example}[theorem]{\bf Example}

{

}
\newcommand{\mm}[1]{\mathrm{#1}}
\newcommand{\mb}[1]{\mathbb{#1}}

\makeatletter
\@namedef{subjclassname@2020}{\textup{2020} Mathematics Subject Classification}
\makeatother

\begin{document}

\title[{\fontsize{7}{7}\selectfont  Loop Space Splittings for Codimensional Sphere Bundles}]{Loop Space Splittings for Codimensional Sphere Bundles}

\author[{\fontsize{7}{7}\selectfont Wen Shen}]{Wen Shen}

\address{College of Mathematics and Physics, Wenzhou University, Wenzhou, P.R.China}
\email{shenwen121212@163.com}

\begin{abstract}
In this paper, we establish two loop decomposition theorems for sphere bundles of real vector bundles under appropriate connectivity and characteristic class assumptions. Specifically, we treat two families of sphere bundles: those arising from rank-$(n-1)$ vector bundles and those from rank-$(n-3)$ vector bundles over closed smooth $n$-manifolds.
\end{abstract}

\maketitle

\section{Introduction}

Loop decomposition equivalence is an increasingly prominent classification paradigm. Formally, two manifolds \(X\) and \(Y\) are equivalent under this relation if their loop spaces satisfy \(\Omega X\simeq \Omega Y\), where \(\Omega\) denotes the loop functor that maps a topological space \(X\) to the space of all based loops in \(X\), and a continuous map \(f:X\to Y\) to the induced map on loops.

 Beben and Theriault \cite{BeTh} proved that two $(n-1)$-connected $2n$-dimensional manifolds are loop-equivalent if and only if they share the same $n$-th Betti number. Parallel results for the odd-dimensional counterpart, namely $(n-1)$-connected $(2n+1)$-manifolds, were established independently by Beben and Wu \cite{BW}, and Huang and Theriault \cite{HT2022}. More recently, Stanton and Theriault \cite{ST} investigated a loop space decomposition for simply-connected Poincar\'e duality complexes of dimension $n$ whose $(n-1)$-skeleton is a co-H-space. A key advantage of these loop decomposition theorems is that they reduce the computation of homotopy groups of highly connected manifolds to that of homotopy groups of spheres and Moore spaces. Related calculations can also be found in the work of Samik Basu and Somnath Basu \cite{Basu,BasuBasu}.

Research on loop decompositions has lately been extended to sphere bundles over $4$-manifolds \cite{Huang23,Huang25}. A notable result in \cite{Huang25} states that the loop spaces of sphere bundles associated to most real vector bundles over simply connected closed $4$-manifolds decompose into products of simpler homotopy types. Shen \cite{Shen} further generalized this conclusion to sphere bundles over $(n-1)$-connected $2n$-Poincar\'e complexes.

This paper continues this line of investigation, focusing on loop space splittings for sphere bundles of real vector bundles over specific manifolds. Motivated by \cite{Theri}, we observe that if a sphere bundle
\begin{equation}
	S^m\stackrel{ i}{\to}  E\stackrel{\pi}{\to} N \label{startbundle}
\end{equation}
  admits a homotopy section $s:N\to E$ (that is, $\pi\circ s$ is a homotopy equivalence), then the  composition
$$\Omega S^m\times \Omega N\stackrel{\Omega i\times \Omega s}{\rightarrow}\Omega E\times \Omega E\stackrel{\mu}{\to }\Omega E$$
induces isomorphisms on $\pi_i$ for $i\ge 0$, thus
is a homotopy equivalence, where $\mu$ is the product of loop space. 

Suppose $N$ is a $n$-manifold. For the bundle \eqref{startbundle}, cell structure analysis shows that a homotopy section always exists when $m\ge n$, but may fail for $m<n$. For $m=n-1$, the bundle \eqref{startbundle} has a homotopy section if its Euler class is trivial. 
In this work, we exploit the existence of homotopy sections to establish our main theorems.  

 We now state our first main theorem.

\begin{theorem}\label{mainth1}
Let $n\ge 5$, and let $N$ be a non-spin, simply connected closed smooth $n$-manifold. For any real spin  vector bundle of rank $n-1$ over $N$, the associated sphere bundle 
$$S^{n-2}\to M\to N$$
 splits after looping, i.e., $\Omega M\simeq \Omega S^{n-2}\times \Omega N$.
\end{theorem}

Before presenting the second theorem, we recall the cohomology group of the classifying space $\mm{BSpin}$. $H^4(\mm{BSpin})\cong \mb{Z}$ is generated by a class $X$ such that $2X=p_1$, where $p_1$ is the first Pontryagin class of the universal bundle over $\mm{BSpin}$, and $X\equiv w_4 \pmod 2$ for the fourth Stiefel-Whitney class  $w_4$. Accordingly, for any spin bundle $\xi$ over a $3$-connected complex, $p_1(\xi)$ is divisible by $2$, and the congruence $w_4(\xi)\equiv \frac{1}{2}p_1(\xi)\pmod 2$ holds.   

\begin{theorem}\label{mainth2}
Let $n\ge 8$, and let $N$ be a $3$-connected closed smooth $n$-manifold satisfying:
\begin{enumerate}
\item either $\frac{1}{2} p_1(N)\equiv \text{1, 5}\pmod 6$,
\item or 
 $\frac{1}{2} p_1(N)\equiv 3\pmod 6$ and $n\equiv 0\pmod 2$.
\end{enumerate}
   If a real vector bundle $\xi$ of rank $n-3$ over $N$ satisfies $\frac{1}{2} p_1(\xi)=0$, then its associated sphere bundle
$$S^{n-4}\to M\to N$$
splits after looping, i.e., 
$\Omega M\simeq \Omega S^{n-4}\times \Omega N$.
\end{theorem}

This paper is organized as follows. Section \ref{Sevec} discusses basic applications of Wu classes and the Wu formula. Sections \ref{Proof1} and \ref{Proof2} present the detailed proofs of Theorem \ref{mainth1} and Theorem \ref{mainth2}, respectively. Indeed, we prove that certain sphere bundles admit homotopy sections (see Lemma \ref{loopspit} and Lemma \ref{loopspit2}) via the coaction perspective introduced in \cite[2.45]{Sw}. Finally, we construct two families of explicit examples to illustrate Theorem \ref{mainth1} and Theorem \ref{mainth2}, respectively. 

\section{Preliminary}\label{Sevec}
In this section, we study Steenrod operations acting on the cohomology groups of smooth manifolds.
 
\begin{lemma}\label{Sq2onN}
	Let $N$ be a closed simply connected non-spin smooth $n$-manifold. Then there exists a generator $x\in H^{n-2}(N)$ such that $$0\ne \mm{Sq}^2\rho_2(x)\in H^n(N;\mb{Z}_2)$$ where $\rho_2$ denotes the mod $2$ reduction.
\end{lemma}
\begin{proof}
We prove this lemma in two steps.

	(1) The mod $2$ reduction $\rho_2:H^{n-2}(N)\to H^{n-2}(N;\mb{Z}_2)$ is surjective. 

This follows from the following exact sequence
$$H^{n-2}(N)\stackrel{\times 2}{\to }H^{n-2}(N)\stackrel{\rho_2}{\to }H^{n-2}(N;\mb{Z}_2)\stackrel{\delta}{\to }H^{n-1}(N)$$
where $H^{n-1}(N)\cong H_{1}(N)\cong \pi_1(N)=0$.

(2) There exists a generator $x\in H^{n-2}(N)$ such that $\mm{Sq}^2\rho_2(x)\ne 0$.

Since $N$ is non-spin, we have $w_2(N)\ne 0$. By the Poincar\'e duality and Item (1), there exists a generator $x\in H^{n-2}(N)$ such that $$0\ne \rho_2(x)\cup w_2(N)\in H^n(N;\mb{Z}_2)$$ 
Recall from \cite[pp.132]{MilnorStasheff} that the second Wu class satisfies $v_2(N)=w_2(N)$.
By the Wu formula, $\mm{Sq}^2\rho_2(x)=\rho_2(x)\cup v_2(N)\ne 0$. This completes the proof.
\end{proof}

\begin{lemma}\label{Sq4onN}
	Let $N$ be a closed smooth $3$-connected $n$-manifold.
	\item (1) If ${p_1(N)}\equiv 0\pmod 3$, then it is trivial that the mod $3$ Steenrod operation
	$\mathcal P^1:H^{n-4}(N;\mb{Z}_3)\to H^n(N;\mb{Z}_3)$. 
	\item (2) If $\frac{1}{2} p_1(N)\equiv 1\pmod 2$, then there is a generator $x\in H^{n-4}(N)$ such that $$0\ne \mm{Sq}^4\rho_2(x)\in H^n(N;\mb{Z}_2)$$ where $\rho_2$ denotes the mod $2$  reduction.
	\item (3) If $\frac{1}{2} p_1(N)\equiv \text{1 or 5}\pmod 6$, then there is a generator $ x\in H^{n-4}(N)$ such that $$0\ne \mm{Sq}^4\rho_2(x)\in H^n(N;\mb{Z}_2),\quad 0\ne \mathcal P^1\rho_3( x)\in H^n(N;\mb{Z}_3)$$
	where $\rho_3$ denotes the mod $3$ reduction.
\end{lemma}
\begin{proof}
By the argument for Lemma \ref{Sq2onN}, the $\rho_2$ and $\rho_3$ are surjective in degree $n-4$. 

 If ${p_1(N)}\equiv 0\pmod 3$, the mod-$3$ Wu class satisfies $$\mathrm v_1(N)\equiv p_1(N)\pmod 3\in H^4(N;\mb{Z}_3)$$ 
 and hence vanishes \cite[pp.229]{MilnorStasheff}. By the Wu formula, we finish the proof.

Let $\mathfrak X=\frac{1}{2} p_1(N)\in H^4(N)$. 

It is well-known that $w_4(N)\equiv \mathfrak X\pmod 2\in H^4(N;\mb{Z}_2)$. Assume $w_4(N)\ne 0$. By  Poincar\'e duality, there exists a generator $x\in H^{n-4}(N)$ such that $$0\ne \rho_2(x)\cup w_4(N)\in H^n(N;\mb{Z}_2).$$ 
The mod-2 Wu class satisfies  $v_4(N)=w_4(N)\in H^4(N;\mb{Z}_2)$ \cite[pp.132]{MilnorStasheff}.
By the Wu formula, $\mm{Sq}^4\rho_2(x)=\rho_2(x)\cup v_4(N)\ne 0$.

Assume $\mathfrak X\equiv \text{1 or 5}\pmod 6$. Since $N$ is $3$-connected, $H^4(N)\cong\mb{Z}^d$ with $d\ge 1$. Let $\{a_i\}_{i=1}^d$ be a basis for $H^4(N)$, and let $[N]\in H_n(N)$ denote the fundamental class of $N$. By Poincar\'e duality, the torsion free part of $H^{n-4}(N)$ admits a basis $\{b_i\}_{i=1}^d$ satisfying $\langle a_i\cup b_j, [N]\rangle=\delta_{ij}$ where $\delta_{ij}$ is the Kronecker delta.

Write $\mathfrak X=\sum_{i=1}^dk_ia_i$ for $k_i\in \mb{Z}$. Since $\mathfrak X\equiv \text{1 or 5}\pmod 6$, we may assume $k_1\equiv 1 \pmod 2$ and $k_2\not\equiv 0 \pmod 3$.

If $k_1\not\equiv 0 \pmod 3$, set $x= b_1\in H^{n-4}(N)$. If $k_2\equiv 1 \pmod 2$, set $x= b_2\in H^{n-4}(N)$.  
 
If $k_1\equiv 0 \pmod 3$ and $k_2\equiv 0 \pmod 2$,
let $g=\mm{gcd}(k_1,k_2)$. Then $\mm{gcd}( g,6)=1$, and there exist $s_1,s_2\in \mb{Z}$ with $\mm{gcd}(s_1,s_2)=1$ such that $s_1k_1+s_2k_2=g$.
Set $x=s_1b_1+s_2b_2$. This element satisfies: \begin{enumerate}
 	\item $x=s_1b_1+s_2b_2$ denotes a generator of $H^{n-4}(N)$;
 	\item $x\equiv 1\pmod 2$ and $x\not\equiv 0\pmod 3$.
 \end{enumerate}
 
 In all cases above, combining the mod-$2$ and mod-$3$ Wu classes with the Wu formula yields $\mm{Sq}^4(\rho_2 x)\ne 0$ and $\mathcal P^1(\rho_3 x)\ne 0$.
\end{proof}


\section{Proof of Theorem \ref{mainth1}}\label{Proof1}

Let $n\ge 5$.
Let $N$ be a smooth simply connected closed $n$-manifold, $N^{(k)}$ be the $k$-skeleton of $N$, $N_{k+1}$ be the cofiber as follows
\begin{equation}
	N^{(k)}\stackrel{}{\to} N\stackrel{}{\longrightarrow }N_{k+1} \label{cofskel}
\end{equation}

Consider the sphere bundle of a rank-$(n-1)$ vector bundle $\xi$
\begin{equation}
	S^{n-2}\to M\stackrel{\pi}{\to} N \label{fibcod2}
\end{equation}
\begin{lemma}\label{fibinN3}
	If $w_2(\xi)=0$, then there exists an $S^{n-2}$-bundle over $N_3$ that fits into a bundle morphism
	 \[
\xymatrix@C=.9cm{
S^{n-2}\ar@{=}[d]^{}\ar[r]^{i}&M\ar[d]^{q}\ar[r]^-{\pi}&N\ar[d]\\
S^{n-2}\ar[r]^{\mm{i}}&E\ar[r]^-{\bar \pi} & N_3
}
\]
where $N_3$ is the cofiber defined in \eqref{cofskel}.
\end{lemma}
\begin{proof}
	If  $w_2(\xi)=0$, then $\xi$ admits a spin structure. Since $\mm{BSpin}_{n-1}$ is $3$-connected, applying the functor $[-,\mm{BSpin}_{n-1}]$ to the cofibration \eqref{cofskel} with $k=2$ yields the result.
\end{proof}

Next, we consider the homotopy type of $N_3$. 
\begin{lemma}\label{cellofN3}
	If $N$ is non-spin, then $N_3$ is $2$-connected and has the homotopy type 
$N^{(n-2)}_3\cup_\alpha e^n$, with the following properties: 
\begin{enumerate}
	\item $N^{(n-2)}_3$ is the $(n-2)$-skeleton of 
	$N_3$ and contains at least one $(n-2)$-cell.	\item $\alpha:S^{n-1}\to N^{(n-2)}_3$ is the attaching map of the top cell.
		\item There exists a generator $y\in H^{n-2}(N_3)$ such that $$0\ne \mm{Sq}^2\rho_2 (y)\in H^n(N_3;\mb{Z}_2).$$
\end{enumerate}
\end{lemma}
\begin{proof}
	The cofibration \eqref{cofskel} for $k=2$ implies  $N_3$ has no cells in dimensions $i=1,2$, so $N_3$ is $2$-connected. The following induced cohomology exact sequence
	 \begin{equation}
	H^{i-1}(N^{(2)})\to H^i(N_{3})\stackrel{p^\ast}{\to} H^i(N)\to H^i(N^{(2)}) \label{cohoexact}
\end{equation}
	  gives $H^{i}(N_3)\cong H^{i}(N)$ for $n-2\le i\le n$. By the universal coefficient theorem,  $H_{n-1}(N_3)=0$ and $H_{n}(N_3)=\mb{Z}$. Since $N$ is simply connected and non-spin, we have $H_2(N)\ne 0$. Poincar\'e duality then gives $H^{n-2}(N)\ne 0$, so $H^{n-2}(N_3)\ne 0$. Thus, $N_3^{(n-2)}$ contains at least one $(n-2)$-cell. The cell structure, i.e. Items (1) (2), now follows from \cite[Proposition 4.1]{Wall1965}.
	  
	  Finally, using Lemma \ref{Sq2onN} and the commutative diagram
\[ 
\xymatrix@C=1cm{
H^{n-2}(N_3)\ar[r]^-{\rho_2}\ar[d]^-{}_-{\cong}&H^{n-2}(N_3;\mb{Z}_2)\ar[r]^-{\mm{Sq}^2}\ar[d]^-{}&H^{n}(N_3;\mb{Z}_2) \ar[d]^{}_-{}\\
H^{n-2}(N)\ar[r]^-{\rho_2}&H^{n-2}(N;\mb{Z}_2)\ar[r]^-{\mm{Sq}^2}&H^{n}(N;\mb{Z}_2)
}
\]
 we complete the proof for Item (3).
	 \end{proof}

\begin{lemma}\label{homotogroup}
	 The homotopy group $\pi_{n-1}(N_3^{(n-2)}\vee S^{n-2})$ is isomorphic to $\pi_{n-1}(N_3^{(n-2)})\oplus \pi_{n-1}( S^{n-2})$ where $\pi_{n-1}( S^{n-2})\cong \mb{Z}_2$.
\end{lemma}
\begin{proof}
	By Lemma \ref{cellofN3}, $N_3^{(n-2)}$ is $2$-connected. The relative Hurewicz theorem implies
	 $$\pi_k(N_3^{(n-2)}\times S^{n-2},N_3^{(n-2)}\vee S^{n-2})=0$$
	 for $k\le n$. Recall that  the homotopy long exact sequence for the pair $(N_3^{(n-2)}\times S^{n-2},N_3^{(n-2)}\vee S^{n-2})$ splits. This finishes the proof.
\end{proof}

We now introduce a key construction:
\begin{lemma}\label{coaction}
	If $N$ is non-spin, there exists a map 
	$$\phi :N^{(n-2)}_3\to N^{(n-2)}_3\vee S^{n-2}$$
	 such that the homotopy class $\phi_\ast( \alpha)\in \pi_{n-1}(N^{(n-2)}_3\vee S^{n-2})$ is  represented by the following components
	 $$(\alpha,1)\in \pi_{n-1}(N^{(n-2)}_3)\oplus \pi_{n-1}(S^{n-2})$$
	 where $\alpha:S^{n-1}\to N^{(n-2)}_3$ is the attaching map of the top cell of $N_3$.
\end{lemma}
\begin{proof}
By Lemma \ref{cellofN3}, $H^{k}(N_3)\cong H^{k}(N_3^{(n-2)})$ for $k\le n-2$. Let $\bar y\in H^{n-2}(N_3^{(n-2)})$ be the class corresponding to the $y\in H^{n-2}(N_3)$ satisfying $0\ne \mm{Sq}^2\rho_2(y)\in H^n(N_3;\mb{Z}_2)$. 

By \cite[Theorem G]{Wall1965}, $N_3^{(n-2)}$ has the following homotopy type
 $$N_3^{(n-2)}\simeq N_3^{(n-3)}\cup e^{n-2}_1\cup \cdots \cup e^{n-2}_d$$
 where the cell $e^{n-2}_1$ represents $\bar y$. Note that $e^{n-2}$ is homeomorphic to $$S^{n-3}\times [0,1]/S^{n-3}\times 0.$$
 By collapsing the $S^{n-3}\times \frac{1}{2}\subset e^{n-2}_1$ to a point, we have the map
 $$\bar \phi:N_3^{(n-2)}\to N_3^{(n-2)}\vee S^{n-2}.$$
Let $p_1$ and $p_2$ be the projections onto the two wedge summands. The composition $p_1\circ \bar \phi$ induces isomorphisms on $H_k$ for $k\ge 0$, thus is a homotopy equivalence. Let $g$ be a homotopy inverse of $p_1\circ \bar \phi$. Define
$$\phi: N_3^{(n-2)}\stackrel{\bar \phi}{\to}N_3^{(n-2)}\vee S^{n-2}\stackrel{g\vee \mm{id}}{\longrightarrow}N_3^{(n-2)}\vee S^{n-2}$$
By Lemma \ref{homotogroup}, 
$$\phi_\ast(\alpha)=(\alpha,\gamma)\in \pi_{n-1}(N_3^{(n-2)}\vee S^{n-2})$$
for some $\gamma\in \pi_{n-1}(S^{n-2})$.

The composition $p_2\circ \phi$ extends to a map
$$T:N_3\simeq N_3^{(n-2)}\cup_\alpha e^n\to Q=S^{n-2}\cup_\gamma e^n $$
 inducing an isomorphism on $H^n$. Note that the $y\in H^{n-2}(N_3)$ lies in the image of $T^\ast:H^{n-2}(Q)\to H^{n-2}(N_3)$. Since $\mm{Sq}^2\rho_2(y)\ne 0$, the operation $\mm{Sq}^2:H^{n-2}(Q;\mb{Z}_2)\to H^{n}(Q;\mb{Z}_2)$ is nontrivial. Therefore, $\gamma=1\in \pi_{n-1}(S^{n-2})$. This completes the proof.
\end{proof}

 We now analyze the cell structure of the total space $E$ of the $S^{n-2}$-bundle over $N_3$ (see Lemma \ref{fibinN3}).
\begin{lemma}\label{cellofEn}
	(1) The $n$-skeleton of $E$ has the homotopy type $$E^{(n)}\simeq (N_3^{(n-2)}\vee S^{n-2})\cup_{ f} e^n$$
	where $f$ denotes the attaching map.
	
	(2) The map $f$ decomposes as
	$$(\alpha,\beta)\in \pi_{n-1}(N_3^{(n-2)}\vee S^{n-2})\cong \pi_{n-1}(N_3^{(n-2)})\oplus \pi_{n-1}( S^{n-2})$$ 
	where $\alpha$ is the attaching map of the $n$-cell of $N_3$.
\end{lemma}
\begin{proof}
	By Lemma \ref{cellofN3}, we have $H^{s}(N_3)=0$ for $s=1,2$, $n-1$ and $H^n(N_3)=\mb{Z}$. Recall the bundle
	$S^{n-2}\to E\stackrel{\bar \pi}{\to} N_3.$
	Since the Euler class of the bundle is zero, there exists a section $N_3^{(n-2)}\to E$. Furthermore, 
	applying the Gysin sequence, we have a ring isomorphism
	\begin{equation}
	H^\ast(E)\cong H^\ast(S^{n-2}\times N_3)\label{cohoiso}.
	\end{equation}
	Therefore, $E^{(n)}$ has the homotopy type
	$$(N_3^{(n-2)}\cup_b e^{n-2})\cup_{f} e^n$$ 
	where the cell $e^{n-2}$ provides a generator of $ H^{n-2}(E)$ inherited from $H^{n-2}(S^{n-2})$; the cell $e^{n}$ provides a generator of $ H^n(E)$ inherited from $H^n(N_3)$. By cellular theorem, $\bar{\pi}$ maps $N_3^{(n-2)}\cup_b e^{n-2}$ to $N_3^{(n-2)}$. We may further assume $\bar{\pi}|_{N_3^{(n-2)}}:N_3^{(n-2)}\to N_3^{(n-2)}$ is the identity. This implies that the attaching map $b$ is null-homotopic. Consequently, it follows that Item (1) of the lemma.

The restriction	 
 $$\bar \pi|_{(n)}:E^{(n)}\simeq (N_3^{(n-2)}\vee S^{n-2})\cup_{ f} e^n \to N_3\simeq N_3^{(n-2)}\cup_{\alpha} e^n$$ collapses the $S^{n-2}$ to a point. This forces the attaching map $f$ to split into the two components $(\alpha ,\beta)$ as claimed.	 
\end{proof}	

\begin{lemma}\label{loopspit}
	If $N$ is non-spin, the $S^{n-2}$-bundle over $N_3$ 
$$S^{n-2}\stackrel{\mm{i}}{\to} E\stackrel{\bar \pi}{\to }N_3$$
 splits after looping, i.e., $\Omega E\simeq \Omega S^{n-2}\times \Omega N_3$.
\end{lemma}
\begin{proof}
	 From Lemma \ref{cellofEn}, we have
	 $$E^{(n)}\simeq (N_3^{(n-2)}\vee S^{n-2})\cup_{ f} e^n$$
	where 
	$f=(\alpha,\beta)\in \pi_{n-1}(N_3^{(n-2)})\oplus \pi_{n-1}( S^{n-2})$. 
	
	If $0=\beta\in \pi_{n-1}(S^{n-2})=\mb{Z}_2$, then 
	$$E^{(n)}\simeq (N_3^{(n-2)}\cup_{\alpha} e^n)\vee S^{n-2}\simeq N_3\vee S^{n-2}$$
	Hence we have an inclusion 
	$$s:N_3\to E^{(n)}\hookrightarrow E$$
	such that $\bar \pi\circ s$ is a homotopy equivalence. Thus, the composition
	$$\Omega S^{n-2}\times \Omega N_3\stackrel{\Omega {\mm{i}}\times \Omega s}{\longrightarrow}\Omega E\times \Omega E\stackrel{\mu}{\to}\Omega E $$
	induces isomorphisms on homotopy groups $\pi_s$ for $s\ge 0$ where $\mu$ denotes the product of loop space. Hence it is a homotopy equivalence.
	
	Assume $1=\beta \in \pi_{n-1}(S^{n-2})=\mb{Z}_2$. 
	By Lemma \ref{coaction}, there exists a map $\phi:N_3^{(n-2)}\to N_3^{(n-2)}\vee S^{n-2}$ such that $$\phi \circ \alpha=(\alpha,1)\in \pi_{n-1}(N_3^{(n-2)}\vee S^{n-2})$$
We extend $\phi$ to a map 
	$$s:N_3\simeq N_3^{(n-2)}\cup_\alpha e^n\to E^{(n)}\simeq (N_3^{(n-2)}\vee S^{n-2})\cup_{(\alpha,1)}e^n\hookrightarrow E$$
	It is easy to check that the composition $\bar \pi\circ s$ is a homology equivalence, thus is a homotopy equivalence. Hence, we get $\Omega E\simeq \Omega S^{n-2}\times \Omega N_3$. 
\end{proof}

\begin{proof}[Proof of Theorem \ref{mainth1}]
	By Lemma \ref{fibinN3}, for the $S^{n-2}$-bundle in Theorem \ref{mainth1}, we have the following bundle morphism
	\[
\xymatrix@C=.9cm{
S^{n-2}\ar[d]^{\mm{id}}\ar[r]^{i}&M\ar[r]\ar[d]^{{q}}&N\ar[d]^-{}\\
S^{n-2}\ar[r]^{ \mm{i}}&E\ar[r]^{\bar \pi}&N_3
}
\]
By Lemma \ref{loopspit}, the map $\Omega \mathfrak i$ has a left homotopy inverse
$r:\Omega E\to \Omega S^{n-2}$. Then $r\circ \Omega {q}:\Omega M\to \Omega S^{n-2}$ is a left homotopy inverse of the map $\Omega i$. Therefore, the $S^{n-2}$-bundle 
$$S^{n-2}\stackrel{i}{\to} M\to N$$
 splits after looping to give
$\Omega M\simeq \Omega S^{n-2}\times \Omega N$.
\end{proof}

\section{Proof of Theorem \ref{mainth2}}\label{Proof2}

Let $n\ge 8$. Let $N$ be a smooth $3$-connected closed $n$-manifold, $N^{(k)}$ be the $k$-skeleton of $N$, $N_{k+1}$ be the cofiber as in \eqref{cofskel}.

Consider the sphere bundle of a rank-$(n-3)$ vector bundle $\xi$
\begin{equation}
	S^{n-4}\to M\stackrel{\pi}{\to} N \label{fibcod4}
\end{equation}
\begin{lemma}\label{fibinN5}
	If $\frac{1}{2} p_1(\xi) =0$, there exists an $S^{n-4}$-bundle over $N_5$ fitting into the following bundle morphism
	 \[
\xymatrix@C=.9cm{
S^{n-4}\ar@{=}[d]^{}\ar[r]^{i}&M\ar[d]^{q}\ar[r]^-{\pi}&N\ar[d]\\
S^{n-4}\ar[r]^{\mm{i}}&E\ar[r]^-{\bar \pi} & N_5
}
\]
\end{lemma}
\begin{proof}
Since $N$ is $3$-connected, then $\xi$ admits a spin structure. From the unstable version of the Postnikov tower in \cite[pp.44]{BuNa}, $\frac{1}{2} p_1(\xi)=0$ guarantees that $\xi$ admits a string structure.  
  Applying the functor $[-,\mm{BString}_{n-3}]$ to \eqref{cofskel} with $k=4$ completes the proof.
\end{proof}

By the same argument for Lemma \ref{cellofN3}, we also have 
\begin{lemma}\label{cellofN5}
	Suppose $\frac{1}{2} p_1(N)\ne 0$. Then $N_5$ is $4$-connected with homotopy type
$N^{(n-4)}_5\cup_\alpha e^n$ satisfying: 
\begin{enumerate}
	\item $N^{(n-4)}_5$ is the $(n-4)$-skeleton of 
	$N_5$ and contains at least one  $(n-4)$-cell.	\item $\alpha:S^{n-1}\to N^{(n-4)}_5$ is the attaching map of the top $n$-cell.
	\end{enumerate}
Moreover, there exists a generator $y\in H^{n-4}(N_5)$ such that: 
\begin{enumerate}
	\item If $\frac{1}{2} p_1(N)\equiv \text{1 or 5}\pmod 6$, then $0\ne \mm{Sq}^4\rho_2 (y)\in H^n(N_5;\mb{Z}_2)$
	and $0\ne \mathcal P^1\rho_3(y)\in H^n(N_5;\mb{Z}_3)$;
	\item If $\frac{1}{2} p_1(N)\equiv 3\pmod 6$, then $0\ne \mm{Sq}^4\rho_2 (y)\in H^n(N_5;\mb{Z}_2)$,
	while $\mathcal P^1:H^{n-4}(N_5;\mb{Z}_3)\to  H^n(N_5;\mb{Z}_3)$ is trivial.
\end{enumerate}
\end{lemma}

Similarly, we have
\begin{lemma}\label{homotogroup2}
	The homotopy group $\pi_{n-1}(N_5^{(n-4)}\vee S^{n-4})$ is isomorphic to $\pi_{n-1}(N_5^{(n-2)})\oplus \pi_{n-1}( S^{n-4})$ where $\pi_{n-1}( S^{n-4})\cong \mb{Z}_8\oplus \mb{Z}_3$.
\end{lemma}

\begin{lemma}\label{coaction2}
Let $\alpha:S^{n-1}\to N^{(n-4)}_5$ be the attaching map from Lemma \ref{cellofN5}.
	 There exists a map 
	$$\phi :N^{(n-4)}_5\to N^{(n-4)}_5\vee S^{n-4}$$
	 such that $\phi_\ast( \alpha)\in \pi_{n-1}(N^{(n-4)}_5\vee S^{n-4})$ decomposes as follows:
	 \begin{enumerate}
	 \item $(\alpha,\ell,\delta)\in \pi_{n-1}(N^{(n-4)}_5)\oplus \mb{Z}_8\oplus \mb{Z}_3$ if $\frac{1}{2} p_1(N)\equiv \text{ 1 or 5}\pmod 6$,
	 	\item $(\alpha,\ell,0)\in \pi_{n-1}(N^{(n-4)}_5)\oplus \mb{Z}_8\oplus \mb{Z}_3$ if $\frac{1}{2} p_1(N)\equiv 3\pmod 6$.
	 \end{enumerate}
	 where  $\ell$ generates $\mb{Z}_8$ and $\delta$ generates $\mb{Z}_3$.
\end{lemma}
\begin{proof}
	The construction of $\phi$ is identical to that in Lemma \ref{coaction}.
	
	 Note that a generator of $\mb{Z}_8\subset \pi_{n-1}(S^{n-4})$ is detected by $\mm{Sq}^4$, a generator of $\mb{Z}_3\subset \pi_{n-1}(S^{n-4})$ is detected by $\mathcal P^1$. The desired component decompositions then follow from Lemma \ref{cellofN5}.  
\end{proof}

We now analyze the cell structure of the total space $E$ of the $S^{n-4}$-bundle over $N_5$ (see Lemma \ref{fibinN5}).
\begin{lemma}\label{cellofEn2}
	(1) The $n$-skeleton of $E$ has a homotopy type $$E^{(n)}\simeq (N_5^{(n-4)}\vee S^{n-4})\cup_{ f} e^n$$
	where $f$ is the attaching map.
	
	(2) $f$ is represented by three components:
	$$(\alpha,\beta,\gamma)\in \pi_{n-1}(N_5^{(n-4)}\vee S^{n-4})\cong \pi_{n-1}(N_5^{(n-4)})\oplus \mb{Z}_8\oplus \mb{Z}_3$$ 
	where $\alpha$ is the attaching map of the $n$-cell of $N_5$. 
	
	(3) In particular, if $n$ is even, $\gamma=0$.
\end{lemma}
\begin{proof}
As the proof for Lemma \ref{cellofEn}, Items (1) and (2) in Lemma \ref{cellofEn2} follow by the bundle structure. 

Applying the Gysin sequence, we have the ring isomorphism
	\begin{equation}
	H^\ast(E;\mb{Z}_3)\cong H^\ast(S^{n-4}\times N_5;\mb{Z}_3).\label{cohoiso2}
	\end{equation}	
	Assume $n$ is even. For any class $x\in H^{n-4}(S^{n-4};\mb{Z}_3)\subset H^{n-4}(E^{(n)};\mb{Z}_3)$, we have $x^2=0$. If $\mathcal P^1 x\ne 0$, the Cartan formula gives 
	$$\mathcal P^1(x^2)=\mathcal P^1(x)\cup x+x\cup \mathcal P^1(x)=(1+(-1)^{n(n-4)})\mathcal P^1(x)\cup x.$$
	By \eqref{cohoiso2} and Item (1), we have $\mathcal P^1(x)\cup x\ne 0$. 
	Since $n$ is even, $$\mathcal P^1(x^2)=2\mathcal P^1(x)\cup x\ne 0$$
	which contradicts $x^2=0$. Therefore $\gamma=0$.
\end{proof}

\begin{lemma}\label{loopspit2}
	If either $\frac{1}{2} p_1(N)\equiv \text{ 1, 5}\pmod 6$, or $\frac{1}{2} p_1(N)\equiv 3\pmod 6$ and $n$ is even, then the $S^{n-4}$-bundle over $N_5$ 
$$S^{n-4}\stackrel{\mm{i}}{\to} E\stackrel{\bar \pi}{\to }N_5$$
 splits after looping, i.e., $\Omega E\simeq \Omega S^{n-4}\times \Omega N_5$.
\end{lemma}
\begin{proof}
	The core of the proof is to construct a homotopy section $$s:N_5\to E.$$
	
	First consider the case $\frac{1}{2} p_1(N)\equiv \text{ 1 or 5}\pmod 6$. By Lemma \ref{coaction2}, there exists a map 
	$$\phi :N^{(n-4)}_5\to N^{(n-4)}_5\vee S^{n-4}$$
	 such that $\phi_\ast( \alpha)=(\alpha,\ell,\delta)\in \pi_{n-1}(N^{(n-4)}_5)\oplus \mb{Z}_8\oplus \mb{Z}_3$.
	  By Lemma \ref{cellofEn2}, the attaching map $f$ for the top cell of $E^{(n)}$ takes the form
	$$f=(\alpha,\beta,\gamma)\in \pi_{n-1}(N_5^{(n-4)}\vee S^{n-4})\cong \pi_{n-1}(N_5^{(n-4)})\oplus \mb{Z}_8\oplus \mb{Z}_3.$$
 There exists $m\in \mb{Z}$ so that $(m\ell,m\delta)=(\beta,\gamma)$. Let $\mathfrak m: S^{n-4}\to S^{n-4}$ be a map of degree $m$ and define
$$\bar s:N^{(n-4)}_5\stackrel{\phi}{\to} N^{(n-4)}_5\vee S^{n-4} \stackrel{\mm{id}\vee \mathfrak m}{\to }N^{(n-4)}_5\vee S^{n-4}.$$	
One checks $\bar s_\ast(\alpha)=(\alpha,\beta,\gamma)$. So $\bar s$ extends to a map $$s: N_5\to E^{(n)}\hookrightarrow E.$$ The composition $\bar \pi\circ s:N_5\to N_5$ is a homology equivalence, hence a homotopy equivalence. 

Next suppose $\frac{1}{2} p_1(N)\equiv 3\pmod 6$ and $n$ is even. By Lemma \ref{coaction2}, there exists a map 
	$$\phi :N^{(n-4)}_5\to N^{(n-4)}_5\vee S^{n-4}$$
	 such that
	  $\phi_\ast( \alpha)=(\alpha,\ell,0)\in \pi_{n-1}(N^{(n-4)}_5)\oplus \mb{Z}_8\oplus \mb{Z}_3$.
	  Lemma \ref{cellofEn2} (3) implies the attaching map $f$ has the form
	$$f=(\alpha,\beta,0)\in \pi_{n-1}(N_5^{(n-4)}\vee S^{n-4})\cong \pi_{n-1}(N_5^{(n-4)})\oplus \mb{Z}_8\oplus \mb{Z}_3.$$
	Repeating the above construction yields the desired section $s$.
\end{proof}

\begin{proof}[Proof of Theorem \ref{mainth2}]
	The conclusion follows from Lemma \ref{fibinN5} and Lemma \ref{loopspit2}. 
\end{proof}

\section{Examples}

\begin{example}
	Let $\mathbb{CP}^{2n}$ denote the complex projective space of complex dimension $2n$ with $n\ge 3$. The total Chern class of its tangent bundle is given by
$$c(\mathbb{CP}^{2n})=(1+x)^{2n+1}\in H^\ast(\mathbb{CP}^{2n}),$$
where $x$ is the first Chern class of the canonical complex line bundle $\xi_{\mathbb{C}}$ over $\mathbb{CP}^{2n}$. Since
$$w_2(\mathbb{CP}^{2n})\equiv c_1(\mathbb{CP}^{2n})\pmod{2}\equiv 1\pmod{2},$$
$\mathbb{CP}^{2n}$ is non-spin. 

For $1\le k\le n-1$, let $\eta_{\mathbb{C}}$ be the Whitney sum of $2k$ copies of $\xi_{\mathbb{C}}$, whose total Chern class is
$$c(\eta_{\mathbb{C}})=(1+x)^{2k}.$$
Let $\eta$ denote the underlying real vector bundle of $\eta_{\mathbb{C}}$. We have
$$w_2(\eta)\equiv c_1(\eta_{\mathbb{C}})\pmod{2}\equiv 0\pmod{2},$$
so that $\eta$ is a real spin bundle of rank $4k$. Finally, set
$$\gamma=\eta \oplus \epsilon^{4n-4k-1}$$
where $\epsilon^{4n-4k-1}$ denotes the trivial bundle of rank $4n-4k-1$. The bundle $\gamma$ satisfies the conditions required in Theorem \ref{mainth1}. 
\end{example}

\begin{example}
Let $\mathbb{HP}^{2n}$ denote the quaternionic projective space for $n\ge 2$. By \cite[Corollary 2.3]{Sz}, the first Pontryagin class of $\mathbb{HP}^{2n}$ satisfies
$$p_1(\mathbb{HP}^{2n})=(4n-2)v\in H^4(\mathbb{HP}^{2n})$$
which yields $\frac{1}{2}p_1(\mathbb{HP}^{2n})\equiv 1,3,5\pmod{6}$. Although the ring structure of $\mathrm{KO}(\mathbb{HP}^{2n})$ is explicitly known \cite[Theorem 3.11]{Sand}, a rank-$(4n-1)$ vector bundle over $\mathbb{HP}^{2n}$ with vanishing $\frac{1}{2}p_1$ cannot be directly constructed from this structure.

For each integer $k\ge 2$, we consider the  cofibration sequence
\begin{equation}
	\mathbb{HP}^{k-1}\to \mathbb{HP}^{k}\stackrel{\rho}{\to} S^{4k}. \label{cofibHP}
\end{equation}
We first have two key facts: $\pi_{4k-1}(\mathrm{BO})=0$ and $\pi_{4k}(\mathrm{BO})=\mathbb{Z}$. 
By analyzing the Atiyah-Hirzebruch spectral sequence for $[\Sigma \mathbb{HP}^{k-1},\mathrm{BO}]\cong \mathrm{KO}^0(\Sigma \mathbb{HP}^{k-1})$, we verify that all relevant $E^{p,-p}_2$-terms are finite abelian groups. Consequently, $[\Sigma \mathbb{HP}^{k-1},\mathrm{BO}]$ itself forms a finite abelian group. 
 This yields the following short exact sequences:
$$
\xymatrix@C=0.7cm{
0 \ar[r]&[S^{4k},\mathrm{BO}]\ar[r]\ar@{=}[d] & [ \mathbb{HP}^{k},\mathrm{BO}]\ar[r] \ar@{=}[d]&[ \mathbb{HP}^{k-1},\mathrm{BO}]\ar@{=}[d]\ar[r]&0\\
0\ar[r] &  [S^{4k},\mathrm{BO}_{4k+1}]\ar[r]^{\rho^\ast}&  [ \mathbb{HP}^{k},\mathrm{BO}_{4k+1}]\ar[r]&  [ \mathbb{HP}^{k-1},\mathrm{BO}_{4k+1}]\ar[r]&0
}
$$

We further analyze the diagram of cofibrations
$$
\xymatrix@C=0.9cm{
&S^{8n-1}\ar[d]^{\alpha}& \\
\mathbb{HP}^{2n-2}\ar[r]&\mathbb{HP}^{2n-1}\ar[r]^{\rho}\ar[d]^{\beta} & S^{8n-4}\\
  &  \mathbb{HP}^{2n}&
}
$$
and apply the functor $[-,\mathrm{BO}_{8n-3}]$ to this diagram.

We next compute $\pi_{8n-1}(\mathrm{BO}_{8n-3})\cong \pi_{8n-2}(O(8n-3))$ via the classical fiber bundle sequence $O(k)\to O(k+1)\to S^{k}$, which induces two long exact sequences of homotopy groups:
$$\pi_{k+2}(S^{k})\to \pi_{k+1}(O(k))\to \pi_{k+1}(O(k+1))\to \pi_{k+1}(S^{k}), \label{k1Ok}$$
$$\pi_{k+2}(S^{k+1})\to \pi_{k+1}(O(k+1))\to \pi_{k+1}(O(k+2))\to \pi_{k+1}(S^{k+1}). \label{k1Ok1}$$
For $k=8n-3$, the kernel of the homomorphism $$\pi_{k+1}(O(k+2))\to \pi_{k+1}(O)$$ is $\mb{Z}_2$ \cite[pp.534]{KervaMilnor}. Combined with $\pi_{8n-2}(O)=0$, we obtain $\pi_{8n-2}(O(8n-1))=\mathbb{Z}_2$. For $k\ge5$, $\pi_{k+2}(S^{k})\cong \pi_{k+1}(S^{k})\cong \mathbb{Z}_2$, so the above exact sequences imply that $\pi_{8n-1}(\mathrm{BO}_{8n-3})$ is a finite abelian group.

By the Freudenthal suspension theorem, for all $n\ge2$, the composite map $\rho\circ \alpha:S^{8n-1}\to S^{8n-4}$ is homotopy equivalent to the suspension $\Sigma f$ of some map $f:S^{8n-2}\to S^{8n-5}$. As established in \cite[pp.126]{WhiteGW1978}, every suspended map is a co-H-map, so the induced homomorphism
$$(\rho\circ \alpha)^\ast=(\Sigma f)^\ast:[S^{8n-4},\mathrm{BO}_{8n-3}]=\mathbb{Z}\to [S^{8n-1},\mathrm{BO}_{8n-3}]$$
is a group homomorphism. This guarantees the existence of infinitely many non-trivial real rank-$(8n-3)$ vector bundles $\xi$ over $S^{8n-4}$ whose pullbacks $(\rho\circ \alpha)^\ast \xi$ are trivial.

Considering the Puppe sequence
$$[\mathbb{HP}^{2n},\mathrm{BO}_{8n-3}]\stackrel{\beta^\ast }{\to} [\mathbb{HP}^{2n-1},\mathrm{BO}_{8n-3}]\stackrel{\alpha^\ast }{\to} [S^{8n-1},\mathrm{BO}_{8n-3}],$$
we obtain vector bundles $\eta$ over $\mathbb{HP}^{2n}$ satisfying $\beta^\ast\eta\cong \rho^\ast \xi$. This isomorphism ensures the vanishing condition $\frac{1}{2}p_1(\eta)=0$. The monomorphism $\rho^\ast$ ensures that $\eta$ is non-trivial. 
\end{example}

\end{document}